\documentclass[12pt]{amsart}
\usepackage{amssymb,verbatim}

\newtheorem{theorem}{Theorem}[section]

\newtheorem{lemma}[theorem]{Lemma}

\theoremstyle{definition}

\newcommand{\tup}{\bigtriangleup}

\newcommand{\ol}{\overline}
\newcommand{\ob}{\ol{B}}

\newcommand{\0}{\emptyset}

\newcommand{\sh}{\mbox{Sh}}

\newcommand{\e}{\varepsilon}

\newcommand{\al}{\alpha}
\newcommand{\ph}{\varphi}
\newcommand{\Ph}{\Phi}
\newcommand{\be}{\beta}
\newcommand{\ga}{\gamma}

\newcommand{\si}{\sigma}
\newcommand{\ta}{\theta}
\newcommand{\om}{\omega}

\newcommand{\da}{\delta}
\newcommand{\nin}{\not\in}
\newcommand{\imp}{\mbox{Imp}}

\newcommand{\dist}{\mbox{dist}}
\newcommand{\diam}{\mbox{diam}}

\newcommand{\C}{\mbox{$\mathbb{C}$}}

\begin{document}

\date{August 22, 2007}
\title[Monotone images of Cremer Julia sets]
{Monotone images of Cremer Julia sets}

\author{Alexander~Blokh}

\thanks{The first author was partially
supported by NSF grant DMS-0456748}

\author{Lex Oversteegen}

\thanks{The second author was partially  supported
by NSF grant DMS-0405774}

\address[Alexander~Blokh and Lex~Oversteegen]
{Department of Mathematics\\ University of Alabama at Birmingham\\
Birmingham, AL 35294-1170}

\email[Alexander~Blokh]{ablokh@math.uab.edu}
\email[Lex~Oversteegen]{overstee@math.uab.edu}


\begin{abstract} We show that if $P$ is a quadratic polynomial with a fixed
Cremer point and  Julia set $J$, then for any monotone map $\ph:J\to
A$ from $J$ onto a locally connected continuum $A$, $A$ is a single
point.
\end{abstract}

\maketitle

\section{Introduction}\label{intro}

Let $P:\C\to\C$ be a complex polynomial of degree $d$ and let $J_P$
be its Julia set. The topological structure of connected Julia sets
$J=J_P$ and the dynamics of $P|_{J}$ have been studied in a number of
papers. The best case, from the topological point of view, is the
case when $J$ is locally connected. Then $J$ is homeomorphic to the
quotient space of the unit circle $S^1/\sim=J_\sim$ with respect to
a specific equivalence relation $\sim$, called an \emph{invariant
lamination}. In this case the map $\sigma:S^1\to S^1$, defined by
$\sigma(z)=z^d$  on the unit circle in the complex plane $\C$,
induces a map $f_\sim:J_\sim\to J_\sim$ which is conjugate to the
restriction $P|_J$. In the following, spaces like $J_\sim$ are called
\emph{topological Julia sets} while the induced maps $f_\sim$ on
them are called \emph{topological polynomials}. Thus, in the locally
connected case, topological polynomials acting on topological
(locally connected) Julia sets are good (one-to-one) models for true
complex polynomials acting on their Julia sets.

Even if $J$ is not locally connected this approach works in many cases. Let $z$
be a periodic point of period $n$ of a polynomial $P$. The point $z$ is called
an \emph{irrational neutral periodic point} if $(P^n)'(z)=e^{2\pi i\al}$ with
$\al$ irrational. In what follows we refer to such points as \emph{CS-points}
(this comes from the fact that all such\ points are either Cremer or Siegel
points). A CS-point $p$ is said to be a \emph{Cremer point} if the power of the
map which fixes $p$ is not linearizable in a small neighborhood of $p$. Suppose
that $P$ is a polynomial with connected Julia set and no CS-points. In his
fundamental paper \cite{kiwi97} Jan Kiwi obtained for such $P$ an invariant
lamination $\sim_P$ on $S^1$ such that $P|_{J_{P}}$ is semi-conjugate to the
induced map $f_{\sim_P}:J_{\sim_P}\to J_{\sim_P}$ by a monotone map $m:J_P\to
J_{\sim_{P}}$ (by \emph{monotone} we mean a continuous map whose point
preimages are connected). In addition Kiwi proved in \cite{kiwi97} that for any
$P$-periodic point $p\in J_P$ the set $J_P$ is locally connected at $p$ and
$m^{-1}\circ m(p)=\{p\}$.

Thus, Kiwi's approach allows one to describe the dynamics of these polynomials
restricted to their Julia sets  by means of a certain monotone map onto a
locally connected continuum. This dynamically motivated monotone map is a
semiconjugacy between the polynomial and the corresponding induced map (in this
case the induced map is a topological polynomial). The aim of this paper is to
show that in some cases the entire approach which uses modeling of the Julia
set by means of a monotone map onto a locally connected continuum breaks down
for topological reasons. By a \emph{basic Cremer polynomial} we mean a
quadratic polynomial $P$ with a \emph{fixed} Cremer point. Our main result is
Theorem~\ref{dege}.

\setcounter{section}{2} \setcounter{theorem}{1}

\begin{theorem}\label{dege} If $P$ is a basic Cremer polynomial
and $\ph:J_P\to A$ is a monotone map onto a locally connected
continuum $A$, then $A$ is a single point.
\end{theorem}

 We show in
Theorem~\ref{dege} that if $P$ is a basic Cremer polynomial then its Julia set
$J_P$ cannot be mapped onto a non-degenerate locally connected continuum by a
monotone map. Thus, in the case of a basic Cremer polynomial, studying the
Julia set by means of a monotone map onto a locally connected continuum is
impossible, and one needs a different approach (see, e.g., \cite{bo06}).

\setcounter{section}{1} \setcounter{theorem}{0}

\section{Main Theorem}

An \emph{unshielded} continuum $K\subset \C$ is a continuum which
coincides with the boundary of the infinite complementary component
of $K$. Given an unshielded continuum $K$ we denote by $R_\al$ the
external (conformal) ray corresponding to the external angle $\al$
and by $\Pi_\al=\ol{R_\al}\setminus R_\al$ the corresponding
principal set (usually, the continuum is fixed in the beginning of
the argument, so we can omit $K$ from the notation; if we do not
want to specify the angle we will omit $\al$ too). A crosscut $C$ of
$K$ is an open arc in $\C\setminus K$ whose closure meets $K$ in two
distinct points. Given an external ray $R$, a crosscut $C$ is said
to be \emph{$R$-transversal} if $\ol{C}$ intersects $\ol{R}$
(topologically transversely) only once; if $t\in R$ then by $C_t$ we
always denote an $R$-transversal crosscut  such that $\ol{C_t}\cap
\ol{R}=\{t\}$. The \emph{shadow of $C$}, denoted by $\sh(C)$, is the
bounded component of $\C\setminus (C\cup K)$. Given an external ray
$R$ we define   the \emph{(induced) order} on $R$ so that $x<_R
y\,(x, y\in R)$ if and only if the point $x$ is ``closer to $K$
\textbf{on the ray $R$} than $y$''.

Our main aim is to prove Theorem~\ref{dege}. However in order to do
so we first prove a geometric Lemma~\ref{dist0} which could be of
independent interest. Given a ray $R$  we call a family of
$R$-transversal crosscuts $C_t$, $t\in R$, an \emph{$R$-defining}
family of crosscuts if for each $t\in R$ there exists an
$R$-transversal crosscut $C_t$ such that $\diam(C_t)\to 0$ as $t\to
K$ and $\sh(C_t)\subset \sh(C_s)$ if $t<_R s$.

\begin{lemma}\label{dist0} If $K$ is an unshielded continuum and $R$
is an external ray to $K$ then there exists an $R$-defining family
of $R$-transversal crosscuts $C_t$, $t\in R$.
\end{lemma}

\begin{proof} Given a point $t\in R$, any $R$-transversal crosscut $C_t$
consists of two semi-open arcs (half-open arcs) connecting $t$ to $K$. On the
uniformization plane one of them will ``grow'' from the point corresponding to
$t$ in the positive (counterclockwise) direction with respect to the ray; such
semi-open arcs will be called \emph{positive arcs at $t$}. Similarly we define
\emph{negative arcs at $t$}. The infimum of the diameters of all positive arcs
at $t$ is denoted by $p(t)$; similarly we define $n(t)$ for negative arcs at
$t$.

By way of contradiction and without loss of generality we may assume
that there exists $\ga>0$ and a sequence $t_i\to K$ in $R$ such that
$n(t_i)>\ga$, $i=1, 2, \dots$. By \cite{miln00} we can choose a
sequence of pairwise disjoint transversal crosscuts $C_{h_i}$,
$h_i\to K$ so that the area of their shadows $\sh(C_{h_i})$ and the
diameters $\diam(C_{h_i})$ converge to $0$. Hence we can find a
crosscut $C_{h_j}=C_j$ so that the area of $\sh(C_j)$ is less than
$\ga^2/99$ and $\diam(C_j)\le \ga/99$. Then the negative ``half'' of
$C_j$, the part of the ray $R$ contained in $\sh(C_j)$, and the set
$K$ enclose an open simply connected domain $U$ on the plane, the
``negative half'' of $\sh(C_j)$.

Choose $t=t_i<_R h_j$. Then the arc length of the subarc $[t,h_j]$
in $R$ is more than $2\ga/3$. Choose a point $x\in U$ so that there
is a straight segment from $t$ to $x$ inside $U$ of length less than
$\ga/9$ (since $R$ is a smooth curve such a segment exists). Consider
all closed balls $\ob$ contained in $\ol{U}$ such that $x\in \ob$.
By compactness, this family contains a ball $\ob=\ob(y, \e)$ of
maximal radius. Set $\partial \ob=S$. Let us show that the set $A=S\cap
\partial U$ has more than one point. Clearly, $A$ is non-empty
(otherwise a ball with the same center and slightly bigger radius
will contain $x$ and will be contained in $\ol{U}$, a
contradiction). Suppose that $A=\{z\}$ is a single point. A tiny
shift of $y$ away from $z$ along the line $zy$ creates a new point
$y'$. We are about to construct a ball centered at $y'$ of  radius
bigger
than $\e$ contained in $\ol{U}$ and containing $x$ which will
contradict the assumptions about $\ob$. Consider two cases.

(1) The angle $\angle xyz$ is obtuse. Consider the ball $\ob'=\ob(y', \e)$. If
$y'$ is sufficiently close to $y$, then $x\in \ob'$. Moreover, the boundary
$S'$ of $\ob'$ consists of two arcs, $L'$ and $L''$, where $L'$ is outside $B$
and $L''\subset B$. Then $L'$ is disjoint from $\partial U$ because it is very
close to the half-circle of $S$ which is cut off $S$ by the diameter of $B$
perpendicular to $yz$ and hence positively distant from $\partial U$. On the
other hand, $L''$ is disjoint from $\partial U$ because $L''\subset B$. Hence
$\ob'\subset U$ and a slightly bigger ball with the same center will contain
$x$ and will be contained in $U$, a contradiction.

(2) The angle $\angle xyz$ is not obtuse. 
Let $H$
be the line segment through $x$ and perpendicular to the segment
$yz$. Then the component  $L$ of $S\setminus H$ \emph{not}
containing $z$ is positively distant from $\partial U$. Let $p\in
S\cap H$. Since the angle of the triangle $\tup y'zp$ at $p$ is
greater than the angle of this triangle at $z$, we see that $d(y',
z)>d(y', p)\ge d(y',x)$. On the other hand, since $\angle xyz$ is
not obtuse then $\angle pyy'$ is not acute, and so $d(p,
y')=\e'>d(p,y)=\e\ge d(x, y)$. Set $B'=B(y', \e')$. As before, the
boundary $S'$ of $\ob'$ consists of two arcs, $L'$ and $L''$, where
$L'$ is outside $B$ and $L''\subset B$. Then $L'$ is disjoint from
$\partial U$ because it is very close to $L$ and $L''$ is disjoint
from $\partial U$ because $L''\subset B$, a contradiction since
$\e'>\e$.

Thus, $\ob$ must intersect $\partial U$ at at least two points.
Since the area of $\sh(C_j)$ is less than $\ga^2/99$ then
$\e<\ga/17$. If there is a point $a\in C_j\cap S$ then there is a
negative arc at $t$ - the concatenation of the straight segment from
$t$ to $x$, the segment inside $\ob$ from $x$ to $a$, and the
appropriate part of $C_j$ - of diameter less than
$\ga/9+2\ga/17+\ga/99<\ga$, a contradiction. If there is a point
$b\in K\cap S$ then there is a negative arc at $t$ - the
concatenation of the straight segment from $t$ to $x$ and the
segment inside $\ob$ from $x$ to $b$ - of diameter less than
$\ga/9+2\ga/17<\ga$, a contradiction. Hence $M=\ob\cap \partial
U=S\cap \partial U\subset R$. On the other hand, by a theorem of
J\text{\o}rgensen (see \cite{jor} and \cite{pom92}) $M$ is
connected. Hence $M$ is a non-degenerate subarc of $S$. Since
$d(S,K)>0$, we can construct another ball $\ob'$ which intersects
$M$ only at its endpoints such that $\ob'\cap K=\0$. Then $\ob'\cap
R$ cannot be connected since $\ob'\cap R$ misses the entire arc $M$
(except for its endpoints) which contradicts the theorem of
J\text{\o}rgensen. Hence $n(t)\to 0, p(t)\to 0$ as $t\to K$ which
shows that there is a family of $R$-transversal crosscuts $C_t$,
$t\in R$, such that $\diam(C_t)\to 0$ as $t\to K$.

This family will be modified  so that $\sh(C_t)\subset \sh(C_s)$ if $t<_R s$.
Observe that
$C_t$ is the union of a negative and a positive arc at $t$. We
modify negative arcs and positive arcs separately  and since it does not matter which
side we consider we denote the one-sided arcs we deal with by $S_t$.
Choose $C_t$ so that for all $s\le_R t$ we have $\diam(C_s)\le \e$
for a small $\e$, follow the ray beyond $t$ towards $K$, and denote
the segment of the ray from $t$ to a point $s\in R$ with $s<_R t$ by
$Q(t, s)$. Let $\Pi$ be the principal set of $R$ and consider two
cases.

(1) Suppose that for every $s, s<_R t$ we have $d(s,t)< 3\e$. By
definition $S_t$ is positively distant from $\Pi$; let
$\da=\dist(S_t, \Pi)>0$ and choose $u<_R t$ so that for all $s\le_R
u$ we have $\diam (S_s)<\min(\e/9, \da/99)$ and $\dist(u,
\Pi)<\min(\e/9, \da/99)$. Then $S_u$ is disjoint from $S_t$ by the
choice of $\da$. Since the ray is smooth, it is easy to see that we
can create a family of short pairwise disjoint arcs $A_v$ from
points $v\in Q(t, u)$ to $S_t$ of diameter less than $4\e$ where
each connector ends at a point $e_v\in S_t$; moreover, these arcs
can be chosen disjoint from $S_u$ and each other and such that
$A_v\cap R=\{v\}$. Denote the union of $A_v$ and the piece of $S_t$
from $e_v$ to $K$ by $S'_v$. Then the family $S'_v, v\in Q(t, u)$
together with $S_t=S'_t$ and $S_u=S'_u$ are the required crosscuts.

 (2) Suppose that there is the first point $u\in R,$ $s<_R
t$ such that $\dist(t, u)=3\e$. Then $S_u$ is disjoint from $S_t$,
and we can proceed the same way as before. That is, we get a family
of negative arcs $S'_v, v\in Q(t, u)$ which together with $S_t=S'_t$
and $S_u=S'_u$ satisfies the second condition of the lemma and
$\diam(S'_v)<5\e$.

Let us proceed with this construction. If  case (1) takes place then
on the next step we replace $\e$ by $\e/9$. If  case (2) takes
place we may need to make several steps until we finally get $u$
such that for all $s\le_R u$ we have $\diam (S_s)<\e/9$. From this
time on we proceed with $\e$ replaced by $\e/9$. Clearly, this way
we complete the construction and thus the proof of the lemma.
\end{proof}

Given an external ray $R_\al$ and an $R_\al$-defining family of
crosscuts $C_t$ one can define the impression by
$\imp(\al)=\cap_{t\in R_{\al}} \ol{Sh(C_t)}$. It can be easily shown
that this definition is equivalent to the standard one and that
$\imp(\al)$ is independent of the choice of the $R_\al$-defining
family of crosscuts \cite{pom92}.

Let us now state a few facts about basic Cremer polynomials $P$
(see, e.g., \cite{grismayeover99}). The notation introduced here
will be used from now on. For convenience, parameterize quadratic
polynomials $P$ as $z^2+v$. Denote the Cremer fixed point of $P$ by
$p$ and the critical point of $P$ by $c$ ($c=0$, however we will
still denote the critical point of $P$ by $c$). Also, denote by
$\si$ the angle doubling map of the circle. It is well-known that if
$P'(p)=e^{2\pi i\rho}$ then there exists a special \emph{rotational}
Cantor set $F\subset S^1$ such that $\si$ restricted on $F$ is
semiconjugate to the irrational rotation by the angle $2\pi \rho$
\cite{bullsent94}; the semiconjugacy $\psi$ is not one-to-one only
on the endpoints of countably many intervals complementary to $F$ in
$S^1$ ($\psi$ maps the endpoints of each such interval into one
point). Of the complementary intervals  the most important one is
the \emph{critical leaf (diameter)} with the endpoints denoted below
by $\al$ and $\be=\al+1/2$ (for definiteness we assume that
$0<\al<1/2$). The limit set $F=\om(\al)$ is exactly the set of
points whose entire orbits are contained in $[\al, \be]$ where the
arc is taken counterclockwise from $\al$ to $\be$.  By Theorem 4.3
of \cite{grismayeover99} we have that $p\in \imp(\ga)$ for every
$\ga\in F$, and $\{p, c, -p\}\subset \imp(\al)\cap \imp(\be)$.

\begin{theorem}\label{dege1} If $P$ is a basic Cremer polynomial
and $\ph:J_P\to A$ is a monotone map onto a locally connected
continuum $A$, then $A$ is a single point.
\end{theorem}

\begin{proof} Set $J=J_P$. By way of contradiction suppose that
$\ph:J\to A$ is a monotone map onto a locally connected
non-degenerate continuum $A$. Since $J$ (and hence all its
subcontinua) is non-separating then by  Moore's Theorem
\cite{moo25} the map $\Ph$, defined on the entire complex plane
$\C$, and identifying precisely \emph{fibers} (point-preimages) of
$\ph$ has $\C$ as its range. This implies that $\Ph(J)=\ph(J)=A$ is
a \emph{dendrite} (locally connected  continuum containing no simple
closed curve). External (conformal) rays $R_\al$ in the $J$-plane
are then mapped into continuous pairwise disjoint curves
$\ph(R_\al)$ in the $A$-plane; below we call the curves $\ph(R_\al)$
\emph{$A$-rays} even though the construction is purely topological.
Clearly, if $R_\al=R$ lands then so does $\ph(R)$ (i.e., $\ph(R)$
converges to a point). Let us show that in fact $\ph(R)$ lands even
if $R$ does not (in which case the principal set $\Pi$ of $R$ is not a
singleton). By Lemma~\ref{dist0} there exists an $R$-defining family
of crosscuts $C_t$. Since $\ph$ is continuous then
$\diam(\ph(C_t))\to 0$ as $t\to J$. Suppose that there is a sequence
$t_n\to J$ such that $\ol{\ph(C_{t_n})}$ is an arc for all $t_n\in
R$ (and hence a crosscut of $A$) and these crosscuts are all
pairwise disjoint. Since $A$ is locally connected then by
Carath\'eodory theory $\ph(C_{t_n})$ converges to a unique point
$x\in A$ which implies that in fact $\ph(C_t)\to x$ as $t\to J$ and
$\ph(R)$ lands. Otherwise denote by $N_t$ the ``negative half'' of
$C_t$. Without loss of generality we may assume that there exists
$t\in R$ such that for all $s<_R t$ in $R$, all $\ol{\ph(N_s)}$ have
the same point, say, $z$, in common, which immediately implies that
$\ph(R)$ lands at $z$.

The union $U$ of $P$-preimages of the points $p$ and $c$ is
countable, and so is the set $\ph(U)$. By Theorem 10.23 of
\cite{nad92} $A$ has countably many branch points. Hence $A$
contains uncountably many \emph{cutpoints} of order 2 which do not
belong to $\ph(U)$, i.e. points $x\nin \ph(U)$ such that $A\setminus
\{x\}$ consists of exactly 2 components. Choose such a cutpoint
$x\in A$ and denote the two components of $A\setminus \{x\}$ by $B$
and $C$. Let us show that there are at least two $A$-rays landing at
$x$ and cutting the entire plane into two half-planes each of which
contains a component of $A\setminus \{x\}$. Indeed, consider
$A$-rays $\ph(R_{\al'})$ and $\ph(R_{\be'})$ landing in $B$. Then
there are two arcs into which $\al', \be'$ divide the circle, and
exactly one of them contains only angles whose $A$-rays land in $B$.
Hence the entire set of angles whose $A$-rays land in $B$ is
contained in an open arc, say, $Q_B$. Similarly, the set of angles
whose $A$-rays land in $C$ is contained in an open arc $Q_C$.
Clearly, $S^1\setminus (Q_B\cup Q_C)$ is the union of two closed
arcs or points, and two angles -  one from each of the components -
would give rise to the desired two rays. Denote these angles by
$\al''$ and $\be''$.

It follows that the fiber $Z=\ph^{-1}(x)$ contains both principal
sets $\Pi_{\al''}$ and $\Pi_{\be''}$. Also, $Z$ cuts $J$ into two
connected sets ($\ph$-preimages of $B$ and $C$). Finally, no forward
$P$-image of $Z$ contains $c$ or $p$. Let us now study the
$P$-trajectory of $Z$. First we show that there exists no $n$ such
that $\si^n(\al'')=\si^n(\be'')\pm 1/2$. Indeed, otherwise $P^n(Z)$
contains $\Pi_{\si^n(\al'')}$ and
$\Pi_{\si^n(\be'')}=-\Pi_{\si^n(\al'')}$.
 Since
$c=0\nin P^n(Z)$ (by the choice of $x$) then there exists $y\in
P^n(Z), y\ne 0$ such that $-y\in P^n(Z)$ too. Then $P|_{P^n(Z)}$ is
not a homeomorphism. By a theorem of Heath (see \cite{hea96}) it
follows that then $P^n(Z)$ \emph{must} contain a critical point, a
contradiction.

Now, given two angles $\ta, \ta'$ we define $d(\ta, \ta')$ as the
length of the shortest arc between $\ta$ and $\ta'$ (we normalize
the circle so that its length is equal to $1$). It is easy to see
that $d(\si(\ta), \si(\ta'))=T(d(\ta, \ta'))$ where $T:[0, 1/2]\to
[0, 1/2]$ is the appropriate scaling of the full tent map. The
dynamics of $T$ shows then that there exists $m$ such that
$d(\si^m(\al''), \si^m(\be''))\ge 1/3$ and by the previous paragraph
we may also assume that $d(\si^m(\al''), \si^m(\be''))<1/2$.
Let
$P'(p)=e^{2\pi i\rho}$ and  $F\subset S^1$ be the rotational
Cantor set  such that $\si$ restricted on $F$ is
semiconjugate to the irrational rotation by the angle $2\pi \rho$.
Since
the longest complementary arcs to the union of two Cantor sets
$F\cup (F+1/2)$ are of length $1/4$, we see that the shorter open arc
complementary to $\si^m(\al''), \si^m(\be'')$ contains points of the
set $F$ (or $F+1/2$) and then since its length is less than $1/2$
the other arc contains points of the same set too. However the
closed connected set $P^m(R_{\al''}\cup Z\cup R_{\be''})$ does not
contain $p$ (or, respectively, $-p$). Choose an angle of $F$ (resp.
$F+1/2$) which belongs to the arc of the circle at infinity
corresponding to the part of the plane not containing $p$ (resp.
$-p$). Then its impression does not contain $p$ (resp. $-p$), a
contradiction.
\end{proof}

\bibliographystyle{amsalpha}
\bibliography{/lex/references/refshort}
\providecommand{\bysame}{\leavevmode\hbox
to3em{\hrulefill}\thinspace}
\providecommand{\MR}{\relax\ifhmode\unskip\space\fi MR }
\providecommand{\MRhref}[2]{%
  \href{http://www.ams.org/mathscinet-getitem?mr=#1}{#2}
} \providecommand{\href}[2]{#2}

\end{document}